\documentclass[11pt]{amsart}
\usepackage{amsthm, amsmath,amsfonts, amssymb, enumerate, textcmds, tensor, enumitem,hyperref}
\usepackage{pst-node}

\usepackage{tikz-cd} 
\usepackage[margin=1.25in]{geometry}
\usepackage{graphicx,tikz}
\usepackage{pinlabel}
\usepackage[colorinlistoftodos]{todonotes}

\newtheorem{thm}{Theorem}[section]
\newtheorem{lem}[thm]{Lemma}
\newtheorem{prop}[thm]{Proposition}

\newtheorem{ex}[thm]{Example}

\DeclareMathOperator{\re}{\textrm{Re}}
\DeclareMathOperator{\im}{\textrm{Im}}
\DeclareMathOperator{\Z}{\mathcal{Z}}

\theoremstyle{definition}
\newtheorem{defn}[thm]{Definition}
\newtheorem{remark}[thm]{Remark}
\begin{document}
\title{A Factorization Theorem for Harmonic Maps}
\author{Nathaniel Sagman}
\address{Mathematics Dept., MC 253-37, Caltech, Pasadena, CA 91125}
\email{nsagman@caltech.edu}
\begin{abstract}
     Let $f$ be a harmonic map from a Riemann surface to a Riemannian $n$-manifold. We prove that if there is a holomorphic diffeomorphism $h$ between open subsets of the surface such that $f\circ h = f$, then $f$ factors through a holomorphic map onto another Riemann surface. If such $h$ is anti-holomorphic, we obtain an analogous statement.
    
    For minimal maps, this result is well known and is a consequence of the theory of branched immersions of surfaces due to Gulliver-Osserman-Royden. Our proof relies on various geometric properties of the Hopf differential.
\end{abstract}
\maketitle
\begin{section}{Introduction}
Let $\Sigma$ be a Riemann surface with a $C^2$ conformal Riemannian metric $\mu$, and let $M$ be a smooth $n$-manifold, $n\geq 2$, equipped with a $C^2$ Riemannian metric $\nu$. Both manifolds are assumed to not have boundary. Harmonic maps $f:(\Sigma,\mu)\to(M,\nu)$ are solutions of the second order semilinear elliptic equation $$\tau(f,\mu,\nu) = \textrm{trace}_\mu \nabla^{\mu^*\otimes f^*\nu} df = 0.$$ On closed manifolds, $\tau(f,\mu,\nu)=0$ arises as the Euler-Lagrange equation of the Dirichlet Energy functional for the metrics $\mu,\nu$. Under fairly general compactness and curvature assumptions on $\Sigma$ and $M$, harmonic maps exist in any non-trivial homotopy class. 

A harmonic map $f:(\Sigma,\mu)\to(M,\nu)$ is admissible if its image is not contained in a geodesic. There is a viewpoint that while admissible harmonic maps are abundant in many contexts, they also reveal rigid geometric properties of the spaces on which they live. The result of this paper is another instance of this phenomenon. It connects local behaviour of a harmonic map to the global complex geometry of the underlying Riemann surface. 

\begin{thm}\label{main}
 Suppose $f:(\Sigma,\mu)\to (M,\nu)$ is an admissible harmonic map, and there is a conformal diffeomorphism $h:\Omega_1\to\Omega_2$ between open subsets of $\Sigma$ such that $f\circ h = f$ on $\Omega_1$. If $h$ is holomorphic, then there is a Riemann surface $(\Sigma_0,\mu_0)$, a holomorphic map $\pi:\Sigma\to \Sigma_0$, and a harmonic map $f_0:(\Sigma_0,\mu_0)\to(M,\nu)$ such that $\pi(\Omega_1)=\pi(\Omega_2)$ and $f$ factors as $f=f_0\circ \pi$. If $h$ is anti-holomorphic, $\Sigma_0$ is a Klein surface and $\pi$ is dianalytic.
\end{thm}

Among other solutions to geometrically flavoured PDEs, Theorem \ref{main} has been known for minimal harmonic maps and pseudoholomorphic curves since the 1970s. Osserman in \cite{O} and Gulliver in \cite{G} studied singularites of the Douglas and Rado solutions to the Plateau problem. The only possible singularities are branch points, which are separated into so-called true branch points and false branch points. Osserman ruled out true branch points and made progress toward the non-existence of false branch points in \cite{O}, and Gulliver showed there are no false branch points in \cite{G}. Together their work proves that the Douglas and Rado solutions are immersed. For an exposition of the Plateau problem, see Chapters 4 and 6 of \cite{CM}.

Curiously, very few properties specific to minimal surfaces come into play in \cite{G}, but rather qualities shared by a larger class of surfaces. This prompted a deeper study of branched immersions of surfaces, which was carried out by Gulliver-Osserman-Royden in \cite{GOR}. A version of Theorem \ref{main} holds for the maps considered in \cite{GOR}. In the next subsection we describe their theory of branched immersions of surfaces and how minimal maps fit into the framework.

Aside from connections to the Plateau problem, the result of Gulliver-Osserman-Royden has other applications. We would like to highlight the work of Moore in \cite{M1} and \cite{M2}, where he studies moduli spaces of minimal surfaces. A map $f$ is somewhere injective if there is a regular point $p$ such that $f^{-1}(f(p)) = p$. Moore uses Theorem \ref{main} for minimal maps to show that a closed minimal map in an $n$-manifold, $n\geq 3$, is not somewhere injective if and only if it factors through a conformal branched covering map. The same result holds for pseudoholomorphic curves \cite[Proposition 2.5.1]{MS}, whose moduli spaces are an active field of study.

If $(\Sigma,\mu)$ is closed with genus at least $2$ and $(M,\nu)$ has negative curvature, then $\Sigma_0$ must have genus at least $2$. The described results for minimal surfaces thus show the somewhere injective condition is generic, for it is very rare for a closed Riemann surface to admit a holomorphic map onto another Riemann surface with non-abelian fundamental group.

In \cite{M1}, \cite{M2}, \cite{MS}, the somewhere injective condition plays a role in various transversality arguments. With this in mind, Theorem \ref{main} should be an essential tool in understanding the distribution of somewhere injective harmonic maps in certain moduli spaces of harmonic maps.

In a different inquiry, Jost and Yau proved a version of Theorem \ref{main} in \cite{Jo} for harmonic maps to K{\"a}hler manifolds, using it as a tool in their study of deformations of Kodaira surfaces. Their work has played a role in the development of the theories of K{\"a}hler manifolds and Higgs bundles. See the survey of Jost \cite{Jo2} for more information. 

\subsection{Minimal surfaces} Loosely following the exposition of Moore in section 4 of \cite{M1}, we explain how the proof of Theorem \ref{main} for minimal maps is deduced from the results in \cite{GOR}. Let $f:(\Sigma,\mu)\to (M,\nu)$ be a $C^1$ map. A point $p\in\Sigma$ is a branch point if $df(p)=0$. We say a branch point is a good branch point of order $m-1$ if there is a choice of coordinates $z$ on $\Sigma$ and $(x_1,\dots, x_n)$ on $M$ such that $f$ is described by the equations $$x_1=\re z^m \hspace{1mm} , \hspace{1mm} x_2=\im z^m \hspace{1mm}, \hspace{1mm} x_k = \eta_k(z) \hspace{1mm} , \hspace{1mm} k\geq 3,$$ where $\eta_k\in o(|z|^m)$. Note that $m=1$ implies we have a regular point.

\begin{remark}
These conventions could be a source of confusion. In \cite{GOR}, Gulliver-Osserman-Royden refer to ``good branch points" as simply ``branch points." This causes no harm in their work, but we should distinguish here.
\end{remark}

In \cite{GOR}, a branched immersion is a map from a surface that is regular everywhere apart from an isolated set of good branch points. For clarity we refer to such a map as a good branched immersion. In this paper, a minimal map is a weakly conformal harmonic map. Gulliver-Osserman-Royden use the representation formula of Hartman and Wintner \cite{HW} to show that a minimal map is a good branched immersion (see Propositions 2.2 and 2.4 in \cite{GOR}). In fact, using this same formula, Micallef and White recover finer coordinate expressions for minimal surfaces (see \cite[Theorem 1.]{MW}). 

An order $m-1$ branch point $p$ of a good branched immersion is ramified of order $r-1$ if $r$ is the maximal non-negative integer such that there is a disk $U$ centered at $p$ on which $f$ factors through a branched covering of degree $r$. If $r=m$, $p$ is called a false branch point, and true otherwise. We say $f$ is unramified if $r=0$. We now recast one of the key results of \cite{GOR}.
\begin{thm}[Proposition 3.19 in \cite{GOR}]\label{gor3} Let $\Sigma$ be a $C^1$ surface, $M$ a $C^1$ manifold, and $f:\Sigma \to M$ a $C^1$ good branched immersion with the unique continuation property and no true branch points. Then there is a $C^1$ surface $\Sigma_0$, a $C^1$ good branched immersion $\pi: \Sigma\to \Sigma_0$, and an unramified $C^1$ good branched immersion $f_0:\Sigma_0\to M$ such that $f=f_0\circ \pi$.
\end{thm}
We do not define the unique continue property of Gulliver-Osserman-Royden (see \cite[page 757]{GOR}), but remark that minimal maps have this property (see \cite[Lemma 2.10]{GOR}). The minimal case is essentially handled in \cite[Proposition 3.24]{GOR}. If a map is conformal, one can dispense of the hypothesis that there are no true branch points, and the objects $\pi$, $\Sigma_0$, and $f_0$ all have the same regularity as $f$ apart from at branch points and images of branch points.

To prove Theorem \ref{gor3}, Gulliver-Osserman-Royden define a relation $\sim$ on $\Sigma$ as follows.
\begin{enumerate}
    \item If $p_1$ and $p_2$ are regular points for $f$, $p_1\sim p_2$ if there exists open sets $\Omega_i$ containing $p_i$, and an orientation preserving $C^1$ map $h:\Omega_1\to\Omega_2$ such that $f\circ h =f$ on $\Omega_1$.
    \item If one of $p_1$ or $p_2$ is a branch point, then in any pair of neighbourhoods $\Omega_i$ containing $p_i$ there exists neighbourhoods $\Omega_i'\subset \Omega_i$ of $p_i$ consisting of only regular points such that for all $p_1'\in \Omega_1'\backslash\{p_1\}$ there exists $p_2'\in \Omega_2'\backslash\{p_2\}$ such that $p_1'\sim p_2'$, and for all $p_2'\in \Omega_2'\backslash\{p_2\}$ there exists $p_1'\in \Omega_1'\backslash\{p_1\}$ such that $p_1'\sim p_2'$.
\end{enumerate}
Gulliver-Osserman-Royden show that this is an equivalence relation and define the quotient $\pi: \Sigma \to \Sigma_0$. They prove $\Sigma_0$ has the structure of a $C^1$ manifold and the map $f_0:\Sigma\to M$ is defined by setting $f_0([p])=f(p)$. Ramification leads to equivalent points, so $f_0$ is unramified.

When $\Sigma$ is a Riemann surface and $\Sigma$ and $M$ are equipped with metrics so that $f$ is minimal, we impose that $h$ is holomorphic. Following the proof of \cite[Proposition 3.24]{GOR}, one can show that the transition maps on $\Sigma_0$ are holomorphic away from the branch points and extend holomorphically via the removeable singularities theorem. One checks in coordinates that the map $f_0$ is minimal with respect to the conformal metric on $\Sigma_0$ obtained via uniformization. The existence of a map $h$ as in Theorem \ref{main} amounts to saying some classes under $\sim$ are not singletons. The minimal case follows directly.

\begin{remark}
Pertaining to the Plateau problem, Gulliver-Osserman-Royden prove a version of Theorem \ref{main} holds for surfaces with particular boundary data. The argument demonstrates that if a false branch point exists, then one can lower the area by passing through a holomorphic map onto another surface. The solutions of Douglas and Rado minimize the area relative to the boundary data, so this cannot occur.
\end{remark}

Gulliver-Osserman-Royden do not consider orientation reversing maps in the definition of $\sim$, but their construction can be modified to allow for this. In this situation, we may end up with a mapping onto a non-orientable surface. Moore notes this in \cite{M1}, although his context is slightly different from ours. Since we could not locate a formal proof in the literature, we explain the necessary adjustments in subsection 4.3.
\subsection{Harmonic maps vs. minimal maps} To prove Theorem \ref{main} in the holomorphic case, we follow the blueprint of Gulliver-Osserman-Royden. That is, we define an equivalence relation on $\Sigma$ and take the quotient as our candidate for the surface $\Sigma_0$. However, it is not obvious how one should define $\sim$. The difficulty comes from the singularities of harmonic maps, in that
\begin{enumerate}[label=(\roman*)]
    \item harmonic maps can have rank $1$ singularities, which do not occur in the theory of Gulliver-Osserman-Royden, and
    \item branch points are not good branch points. At best, we can combine the Hartman-Wintner formula with \cite[Lemma 2.4]{C} to see that near a branch point $p$ of order $m-1$ there is a $C^1$ coordinate $z$ on the source and a $C^\infty$ coordinate on the target such that $p\mapsto 0$, $f(p)\mapsto 0$, and $f$ may be expressed $f=(f^1,\dots, f^n)$ with $$f^1 = p^1 \hspace{1mm} , \hspace{1mm} f^k = p^k + r^k \hspace{1mm} , \hspace{1mm} k\geq 2,$$ where $p^1$ is a spherical harmonic of degree $m$, $p^k$ is a spherical harmonic of degree at least $m$, and $r^k\in o(|z|^m)$.
\end{enumerate}
To overcome these difficulties, we exploit the geometry of a holomorphic quadratic differential known as the Hopf differential. In some sense, the Hopf differential treats rank $1$ and $2$ points on an equal footing. Thus, if we define $\sim$ in terms of a condition on the Hopf differential, in theory we shouldn't encounter any difficulties due to rank $1$ singularities. In practice this is mostly true--at some points we need to refer to the Hartman-Wintner formula. As for (ii), the Hopf differential defines a ``natural coordinate" for the harmonic map near a branch point, in which the geometry can be more easily probed. At a false branch point, we see ramification behaviour similar to that displayed by minimal maps.

The only missing piece of Gulliver-Osserman-Royden's theory is the unique continuation property. In Proposition \ref{cont}, we show that analytic continuation of natural coordinates for the Hopf differential induces a continuation of $h$. Using this proposition, we establish a ``holomorphic unique continuation property" (Proposition \ref{hUCP}).

\subsection{Future directions.} It is tempting to conjecture that some version of 
 Theorem \ref{main} should hold without the hypothesis that $h$ is conformal. The main motivation would be to improve our understanding of somewhere injective harmonic maps. We would like to point out that, in view of the example below, we cannot expect the map $\pi$ to be holomorphic with respect to a complex structure on $\Sigma_0$. 
\begin{ex}
Let $(\Sigma_0,\mu_0)$ be a closed hyperbolic surface and $f_0:(\Sigma_0,\mu_0)\to (M,\nu)$ a totally geodesic map. Fix a smooth surface $\Sigma$ of genus at least $2$ and a homotopy class of maps $\mathbf{f}:\pi_1(\Sigma)\to\pi_1(\Sigma_0)$ with degree at least $2$. Any $C^2$ metric $\mu$ yields a unique harmonic map $\pi:(\Sigma,\mu)\to (\Sigma_0,\mu_0)$ in the homotopy class $\mathbf{f}$. One can then find many diffeomorphisms $h:\Omega_1\to \Omega_2$ between open subsets of $\Sigma$ such that $f\circ h = f$, and by construction $f$ factors as $f=f_0\circ \pi$. Generically, the surface $(\Sigma,\mu)$ will not admit a holomorphic map onto any Riemann surface of genus at least $2$.
\end{ex}
We simplify our study of singularities using complex analytic methods. Without the conformal hypothesis, the only local information we have comes from the Hartman-Wintner representation formula. If this is the main tool, then it is also natural to ask about more general solutions to second order semiliinear elliptic systems, rather than just harmonic maps. An analysis of singularities would be related to understanding local behaviour of spherical harmonics.

A substitute for the unique continuation property seems to be a large hurdle. Implicit in the proof of the unique continuation property for minimal maps is the following result (see \cite[Lemma 2.10]{GOR}).
\begin{prop}\label{UCP}
Let $\mathbb{D}\subset \mathbb{R}^2$ be the unit disk. Suppose $u_1,u_2:\mathbb{D}\to M$ are minimal maps such that, for all open sets $D_1\subset \mathbb{D}$ containing $0$, there is an open subset $D_2\subset\mathbb{D}$ (possibly not containing $0$) such that $u^2(D_2)\subset u^1(D_1)$. Then there exists an open subset $D'\subset\mathbb{D}$ containing $0$ such that $u^2(D')\subset u^1(\mathbb{D})$.
\end{prop}
This result above fails emphatically if we replace minimal maps with harmonic maps, even if $M=\mathbb{R}^2$. Indeed, the simple example $$u_1(x,y) = (x,xy) \hspace{1mm} , \hspace{1mm} u_2(x,y) = (x,y)$$ does not satisfy Proposition \ref{UCP}. On the other hand, our ``holomorphic unique continuation property" provides a substitute for Proposition \ref{UCP} (see Proposition \ref{hUCP}). This is one of the reasons we expect a more general version of Theorem \ref{main} to be much more delicate, and we defer this investigation to a future project. 
\subsection{Acknowledgements} Many thanks to Vlad Markovic for encouragement and sharing helpful ideas. I would also like to thank John Wood and J{\"u}rgen Jost for comments on earlier drafts.
\end{section}

\begin{section}{Harmonic Maps from Riemann Surfaces}
We give background on harmonic maps. The content is standard and can be found in any text on harmonic maps. We then discuss analytic continuation and singularities.
\subsection{Harmonic maps.} Throughout, we let $(\Sigma,\mu)$ denote a Riemann surface with a $C^2$ conformal metric, and $(M,\nu)$ an $n$-manifold, $n\geq 3$, with a $C^2$ Riemannian metric. 
A $C^2$ map $f:(\Sigma,\mu)\to (M,\nu)$ gives a pullback bundle $f^*TM$, and the derivative $df$ defines a section of the endomorphism bundle $T^*\Sigma\otimes f^*TM$. We denote by $\nabla$ the connection on the tensor bundle $T^*\Sigma\otimes f^*TM$ induced by the Levi-Civita connections $(\nabla^\mu)^*$ and $\nabla^{f^*\nu}=f^*\nabla^\nu$ on $T^*\Sigma$ and $f^*TM$ respectively. 

\begin{defn}
The tension field of a $C^2$ map $f:(\Sigma,\mu)\to (M,\nu)$ is the section of $f^*TM$ given by $$\tau=\tau(f,\mu,\nu) = \textrm{trace}_\mu \nabla df.$$ The map $f$ is harmonic if $\tau= 0 $.
\end{defn}
In a local conformal coordinate $z=x+iy$, the tension field is given by
$$\tau = |\mu|^{-1}\Big (\nabla_{\frac{\partial}{\partial x}}^{f^*\nu}df\Big (\frac{\partial}{\partial x}\Big ) + \nabla_{\frac{\partial}{\partial y}}^{f^*\nu}df\Big (\frac{\partial}{\partial y}\Big )\Big )$$
and hence $\tau(f,\mu,\nu)=0$ defines a conformally invariant semilinear elliptic equation of second order. If $\mu$ is $C^\alpha$ and $\nu$ is $C^\beta$, then the harmonic map is $C^\gamma$, where $\gamma=\min\{\alpha+2,\beta+1\}$. We are also allowing $\alpha,\beta=\infty$ or $\omega$. Therefore, our harmonic maps are at least $C^3$.
\begin{defn}
A harmonic map $f:(\Sigma,\mu)\to (M,\nu)$ is minimal if it is also weakly conformal. That is, conformal in the sense of distributions.
\end{defn}
As for the complex theory, we set $(f^*TM)^\mathbb{C}=f^*TM\otimes \mathbb{C}$ to be the complexification of the pullback bundle and extend $f^*\nu$ linearly. Given a local coordinate $z=x+iy$, define $$\frac{\partial}{\partial z} = \frac{1}{2}\Big ( \frac{\partial}{\partial x} - i\frac{\partial}{\partial y}\Big ) \hspace{1mm} , \hspace{1mm} \frac{\partial}{\partial \overline{z}} = \frac{1}{2}\Big ( \frac{\partial}{\partial x} + i\frac{\partial}{\partial y}\Big ).$$ Using this coordinate, we have locally defined sections of $(f^*TM)^\mathbb{C}$ given by $$f_z = df\Big ( \frac{\partial}{\partial z}\Big ) \hspace{1mm} , \hspace{1mm} f_{\overline{z}} = df\Big ( \frac{\partial}{\partial \overline{z}}\Big ).$$ 

\begin{defn}
The Hopf differential of a harmonic map $f:(\Sigma,\mu)\to (M,\nu)$ is the holomorphic quadratic differential $\Phi$ on $\Sigma$ specified by the family of local expressions $$\langle f_z, f_z\rangle_{f^*\nu} dz^2,$$ where $z$ ranges over local coordinates for $\Sigma$.
\end{defn}
It follows from the definitions that $f$ is minimal if and only if $\Phi = 0$. If $\Phi$ does not vanish identically, the zeros of $\Phi$ are independent of the parametrization and discrete. If $\Phi(p)\neq 0$, then near $p$ we can find a holomorphic coordinate $z$ such that $$\Phi(z) = dz^2.$$ If $\Phi(p)=0$, there is a coordinate $z$ such that $$\Phi(z) = z^n dz^2.$$ Such coordinates are called natural coordinates for $\Phi$. 

The quadratic differential $\Phi$ induces a singular flat metric: the $\Phi$-metric. Locally, if $\Phi=\phi(z)dz^2$, then the metric tensor is $$|\phi(z)||dz|^2.$$  The singular points are the zeroes of $\Phi$. A disk of radius $r$ centered at a point $p$ in the $\Phi$-metric shall be called a $\Phi$-disk and written $B_r(p)$. The induced distance function is denoted $d(\cdot,\cdot)$. Although we work with different differentials in the course of the paper, the use of this notation in context should be clear.

\subsection{Analytic continuation} Until Section 4, let  $f:(\Sigma,\mu)\to (M,\nu)$ be an admissible harmonic map with non-zero Hopf differential $\Phi$ and let $h:\Omega_1\to\Omega_2$ be a holomorphic map as in the statement of Theorem \ref{main}. We treat anti-holomorphic maps in Section 4. Throughout the paper, we let $\mathcal{Z}$ denote the zero locus of $\Phi$. By restricting, we assume $\Omega_1$ is a $\Phi$-disk. 

We use the geometry of the Hopf differential to analytically continue $h$. Let $p\in\Omega_1$ be such that $\Phi(p)\neq 0$, and let $U\subset \Omega_1$ be an open subset containing $p$ such that $\Phi\neq 0$ in $U$. Given a holomorphic local coordinate $z$ in $U$, we define a local coordinate $w$ on $h(U)$ by $w=z\circ h^{-1}$. In these coordinates, $h$ is given by $w(h(z))=z$ and $$df_p\Big ( \frac{\partial}{\partial z}\Big ) = df_{h(p)}\Big ( \frac{\partial}{\partial w}\Big )\in T_{f(z)}M\otimes \mathbb{C}.$$ 
 \begin{remark}
 Here we are viewing $df$ as a map from $T\Sigma\to TM$ rather than as a section of the endomorphism bundle $T^*\Sigma\otimes f^*TM$.
 \end{remark}
 Choosing $z$ to be a natural coordinate with $z(p)=0$, we obtain $$\langle f_w, f_w\rangle (w(h(z)))=\langle f_z,f_z\rangle(z) = 1.$$ Therefore, $w$ defines a natural coordinate on $h(U)$. We have proved the following lemma.

\begin{lem}\label{isometry}
$h$ is a local isometry in the $\Phi$-metric. If $\Omega_1$ is a $\Phi$-disk then so is $\Omega_2$, and $h$ takes a natural coordinate $z$ on $\Omega_1$ to a natural coordinate $w$ on $\Omega_2$ in which $w(h(z))=z$.
\end{lem}
The goal of this subsection is to prove the proposition below. In the proof we use the notion of a maximal $\Phi$-disk. See section 5 in \cite{St} for a detailed discussion on maximal $\Phi$-disks. Let $\mathcal{Z}$ denote the zero set of $\Phi$ (which is isolated).
\begin{prop}\label{cont}
Suppose $\Omega_1,\Omega_2$ are $\Phi$-disks with no zeros of $\Phi$ and that $\gamma:[0,L]\to \Sigma$ is a curve starting in $\Omega_1$ and that $\gamma$ first strikes $\partial \Omega_1$ at a point $q$. If there is an $\epsilon> 0$ such that $$\textrm{min}\big \{\inf_{s\in\gamma|_{\Omega_1},t\in\mathcal{Z}}d(s,t),\inf_{s\in \gamma|_{\Omega_1},t\in\mathcal{Z}}d(h(s),t)\big \}\geq \epsilon$$ then there is a neighbourhood of $q$ in which $h$ can be analytically continued along $\gamma$.
\end{prop}
\begin{proof}
We can choose an arc on $\partial\Omega_1$ centered at $q$ on which $\Phi\neq 0$. We then connect the endpoints via an arc contained inside $\Omega_1$ so that the enclosed region $U$ is a topological disk. We pick these arcs in such a way that $$\textrm{min}\big \{\inf_{s\in U,t\in\mathcal{Z}}d(s,t),\inf_{s\in U,t\in\mathcal{Z}}d(h(s),t)\big \}\geq \epsilon/2.$$
The restriction of the $\Phi$-metric to any compact region that does not intersect $\mathcal{Z}$ is complete. As $h$ is an isometry in the $\Phi$-metric, we can extend it to a map $h:\overline{U}\to \overline{U}$. Therefore, we have a well-defined point $h(q)$.

For every point $p\not \in \mathcal{Z}$, there is a maximal radius $r_p$ such that we can extend any natural coordinate centered at $p$ to a $\Phi$-disk of radius $r_p$. $r_p$ does not depend on the initial choice of natural coordinate. If $d(s,t)=\delta$, then $$r_s-\delta \leq r_t\leq r_s + \delta.$$ Let $r_0=\textrm{min}\{r_q,r_{h(q)}\}$. Select a point $q'\in B_{r_0/4}(q)\cap\Omega_1$. This point satisfies $r_{q'}\geq 3r_0/4$ and likewise for $h(q')$. Let $\delta=d(q,q')$ and take a natural coordinate $z$ in a $\Phi$-disk $B_{\delta/2}(q')$. We restrict $h$ to this $\Phi$-disk, and as above, we use $h$ to build a natural coordinate $w$ on $B_{\delta/2}(h(q'))$. More precisely, we have a disk $D\subset \mathbb{C}$ of radius $\delta/2$ and two holomorphic maps $$\varphi: D\to B_{\delta/2}(q') \hspace{1mm} , \hspace{1mm}\psi: D\to B_{\delta/2}(h(q'))$$ such that $z=\varphi^{-1}$, $w=\psi^{-1}$. We can extend these maps to a larger disk $D'\subset \mathbb{C}$ with radius $3r_0/4$. The map $$w^{-1}\circ z:  B_{3r_0/4}(q')\to B_{3r_0/4}(h(q'))$$ is a holomorphic diffeomorphism that agrees with $h$ on $B_{\delta/2}(q')$. Since $B_{r_0/2}(q)\subset B_{3r_0/4}(q')$, we see we have analytically continued $h$ to the open set $\Omega_1\cup B_{r_0/2}(q)$. From conformal invariance, the map $f\circ h$ is harmonic, and hence the Aronszajn theorem \cite{Ar} implies $f\circ h = f$ on $\Omega_1\cup B_{r_0/2}(q)$.
\end{proof}
Via this result, we often find ourselves in the following situation: either $h$ can be continued along an entire curve $\gamma$, or we have a segment $\gamma'\subset\gamma$ along which $h$ has been continued but the endpoint of $h(\gamma')$ is a zero of $\Phi$.

We remark that there is no guarantee that the analytic continuation is a diffeomorphism. It is at least a local diffeomorphism and a local isometry for the $\Phi$-metric.

\subsection{Harmonic singularities.}
Toward the proof of the main theorem, we rule out possible pathological behaviour of harmonic maps near rank $1$ singularities. We need not delve too deep into the theory of singularities, but we invite the reader to see Wood's thesis \cite{Wood} and the paper \cite{wood2}, in which he studies singularities of harmonic maps between surfaces in detail.

Our key tool is the Hartman-Wintner theorem \cite{HW}, which gives a local representation formula for harmonic maps. Let $z$ be a holomorphic coordinate centered on a disk centered at $p\in \Sigma$ with $z(p)=0$, and let $(x_1,\dots, x_n)$ be normal (but not necessarily orthogonal) coordinates in a neighbourhood $U$ of $f(p)$ such that $f(p)=0$. According to the Hartman-Wintner theorem, we can write the components $(f^1,\dots, f^n)$ as $$f^k = p^k + r^k$$ where $p^k$ is a spherical harmonic (a harmonic homogeneous polynomial) of some degree $m<\infty$ and $r^k\in o(|z|^m)$. We are allowing $p^k=\infty$, which means $f^k=0$.

By permuting the coordinates, we may assume $\deg p^1 = \min_k \deg p^k$, and $\deg p^k\geq \deg p^2$ for all $k\geq 3$. Note $\deg p^1,\deg p^2<\infty$, for otherwise Sampson's result \cite[Theorem 3]{S} implies $f$ takes its image in a geodesic. 
\begin{lem}\label{lonely}
There does not exist a sequence of points $(p_n)_{n=1}^\infty\subset \Sigma$ converging to $p$ with the property that there exists a (not necessarily conformal) diffeomorphism $h_n$ taking a neighbourhood of $p_n$ to a neighbourhood of $p$ that leaves $f$ invariant.
\end{lem}
\begin{proof}
Arguing by contradiction, suppose there is such a sequence $(p_n)_{n=1}^\infty$. Since $f$ is an embedding near regular points, $p$ must be a singular point. Choose a coordinate $z$ on the source and normal coordinates on the target with $p=0$, $f(p) =0$. We apply Hartman-Wintner to obtain the formula $$f^k = p^k + r^k$$ with the same degree assumptions as above. It is clear that there is at least one $p^k$ with $\deg p^k = m >1$, $m\neq \infty$.

We invoke a result of Cheng \cite[Lemma 2.4]{C}: there is a $C^1$ diffeomorphism from a neighbourhood of $0$ in $\mathbb{R}^2$ to a neighbourhood of $p$, taking $0$ to $0$ in coordinates, and such that $$f^k\circ \varphi(w) = p^k(w)$$ As a spherical harmonic of degree $m$, the zero set of $p^k$ consists of $m$ distinct lines going through the origin, arranged so that the angles between two adjacent lines is constant (this is an easy consequence of homogeneity). Notice that in our neighbourhood of $p$, $$\{q: f^k(q) = f^k(p)\}= \{\varphi(w) : p^k(w) = p^k(0)\}.$$ Therefore, the set $\{q: f^k(q) = f^k(p)\}$ is  collection of $m$ disjoint $C^1$ arcs all transversely intersecting at the origin. For $n$ large enough, $p_n$ lies inside the coordinate chart determined by $\varphi$, and hence it lies on one of the arcs. Fixing such a $p_n$, we use that $h_n$ is a diffeomorphism to see that there should be $m-1$ more curves transversely intersecting the line containing $p_n$, and such that $f(q)=f(p)$ on those curves. This is a clear contradiction.
\end{proof}
\end{section}

\begin{section}{Holomorphic Factorization}
Throughout this section, we continue to assume $h:\Omega_1\to\Omega_2$ is a holomorphic diffeomorphism. Following the structure of Section 3 in \cite{GOR}, we prove Theorem \ref{main} holds for such $h$ (although the technical details of our proofs are for the most part quite different). 
\subsection{The equivalence relation}
\begin{defn}\label{sim} Given $p_1,p_2\in \Sigma$, we define a relation $\sim$ by
\begin{enumerate}
    \item If $p_1,p_2\not \in \Z$, $p_1\sim p_2$ if there exists open sets $\Omega_1,\Omega_2$ such that $p_i\in \Omega_i$ and a holomorphic diffeomorphism $h:\Omega_1\to \Omega_2$ such that $f=f\circ h$ on $\Omega_1$.
    \item If one of $p_1,p_2$ is a zero of $\Phi$, then for any pair of neighbourhoods $\Omega_i$ containing $p_i$ one can find smaller neighbourhoods $\Omega_i'\subset \Omega_i$ containing $p_i$ such that for each $q_1\in\Omega_1'\backslash \{p_1\}$ there exists $q_2\in\Omega_2'\backslash\{p_2\}$ such that $q_1\sim q_2$, and for each $q_2\in \Omega_2'\backslash\{p_2\}$ there is a $q_1\in \Omega_1'\backslash \{p_1\}$ such that $q_2\sim q_1$.
\end{enumerate}
\end{defn}
 If $p_1\sim p_2$ then $f(\Omega_1')=f(\Omega_2')$ and $f(p_1)=f(p_2)$ are apparent from the definition. Recall $\mathcal{Z}=\{p\in \Sigma: \Phi(p)\neq 0\}$.
 \begin{prop}\label{equiv}
$\sim$ is an equivalence relation.
\end{prop}
\begin{proof}
Reflexivity and symmetry are obvious. As for transitivity, this is clear if $p_1,p_2,p_3$ are all not zeros of $\Phi$. If at least one is a zero, we consider two cases:
\begin{enumerate}[label=(\roman*)]
    \item $p_1,p_3$ are zeros, or
    \item $p_2$ is a zero while $p_1,p_3$ are not
\end{enumerate}
The other cases are trivial. Case (i) can be seen from the definitions: take $\Omega_1,\Omega_2$ containing $p_1,p_2$ respectively such that for all $p_1'\in\Omega_1\backslash\{p_1\}$ there exists $p_2'\in\Omega_2\backslash\{p_2\}$ with $p_1'\sim p_2'$. Within $\Omega_2$ we find an open set $\Omega_2'$, and then an open set $\Omega_3'$ containing $p_3$ with the same property. Set $$\Omega_1'=\{p_1'
\in\Omega_1\backslash\{p_1\}: \textrm{ there exists }p_3'\in\Omega_3' \textrm{ such that }p_1'\sim p_3'\}\cup \{p_1\}.$$ We can find an open disk centered at $p_1$ inside $\Omega_1'$ by applying the definition of $\sim$ to the open sets $\Omega_1,\Omega_2'$. It is also clear that $\Omega_1'$ is open away from $p_1$, and hence it is open. It is now simple to check that $\Omega_1'$ and $\Omega_3'$ satisfy the definition of $\sim$.

The second case requires more work. Select $\Phi$-disks $U_1,U_3$ of radius $R>0$ around $p_1$ and $p_3$ respectively such that there are no points $q_i$ with $q_i\sim p_i$ and no zeros of $\Phi$. Let $U_1',U_3'$ be $\Phi$-disks centered at the same points with half the radius. Using $\sim$, we can find open sets $\Omega_i\subset U_i'$ containing $p_i$ such that $f(\Omega_3)\subset f(\Omega_1)$ and every point in $q\in\Omega_3\backslash\{p_3\}$ is equivalent to a point in $\Omega_1\backslash\{p_1\}$. We shrink $\Omega_3$ to turn it into an open disk in the $\Phi$-metric centered at $p_3$ with radius $\delta<R/2$.

Let $p_i'\in \Omega_i$ be such that $p_3'\sim p_1'$. Viewing $\Omega_3$ in natural coordinates, let $\gamma$ be the straight line from $p_3'$ to $p_3$. We have a holomorphic map $h$ taking a neighbourhood of $p_3'$ to one of $p_1'$ that leaves $f$ invariant. We analytically continue along $\gamma$ as much as we can. Either $h(\gamma)$ hits a zero of $\Phi$ or we can continue up until the endpoint. The $\Phi$-length of any segment of $h(\gamma)$ is at most $\delta$, and we infer $h(\gamma)$ is contained in $B_{R/2+\delta}(p_2)\subset U_3$. Thus, $h(\gamma(t))$ can never be a zero for any time $t$, and we can continue to the endpoint. From the proof of Proposition \ref{cont}, $p_3=\gamma(1)$ is equivalent to the endpoint $h(\gamma(1))$.

To finish the proof, we need to argue $h(\gamma(1))=p_1$. Let $q_1=h(\gamma(1))$. We do know $p_3\sim q_1$. We claim we could have chosen $R$ small enough to ensure no point other than possibly $p_1$ is equivalent to $p_3$. Indeed, if this is not possible, then we get a sequence of points $(q_n)_{n=1}^\infty$ converging to $p_1$ such that $p_3\sim q_n$ for all $n$. Using transitivity of $\sim$ for points in $\Sigma\backslash\Z$, we can then construct a sequence of points $q_n'$ converging to $p_3$ that are all equivalent to $p_3$. This directly contradicts Lemma \ref{lonely} and completes the proof.
\end{proof}
We use Proposition \ref{equiv} to prove another useful property of $\sim$.

\begin{lem}\label{nofriends}
Suppose $p_1,p_2\not \in \Z$. Then there is no sequence $(q_n)_{n=1}^\infty$ such that $q_n\sim p_1$ for all $n$ and $q_n\to p_2$ as $n\to\infty$.
\end{lem}
\begin{proof}
Again going by way of contradiction, assume such a sequence $q_n$ exists. Firstly, by Lemma \ref{lonely}, we cannot have $p_1\sim p_2$. Using the definition of $\sim$, we see that in any pair of neighbourhoods $\Omega_i$ of $p_i$, we can find points $p_i'\in \Omega_i$ such that $p_1'\sim p_2'$. 

Let $\delta,\epsilon>0$ and $\tau = \epsilon+2\delta$. We choose $\delta,\epsilon$ to be small enough to ensure
\begin{enumerate}[label=(\roman*)]
    \item there is no point equivalent to $p_1$ in $B_\tau(p_1)\backslash\{p_1\}$,
    \item there is no point equivalent to $p_2$ in $B_\delta(p_2)\backslash\{p_2\}$, and
    \item there are no zeros of the Hopf differential in either ball.
\end{enumerate}
Choose $p_1'\in B_\epsilon(p_1)$ that is equivalent to a point $p_2'\in B_\delta(p_2)$. In natural coordinates, let $\gamma$ be the straight line path from $p_2'$ to $p_2$. $\gamma$ has length at most $\delta$, and hence the image of any segment of $\gamma$ along an analytic continuation of $h$ lies in $B_\tau(p_1)$. Thus, we can continue $h$ along $\gamma$ as much as we like, and we extend to the boundary point $p_2$. The endpoint $h(\gamma(1))$ is then equivalent to $p_2$. Since $p_1\not\sim p_2$, $h(p_2)\neq p_1$. 

Set $q_1'=h(p_2)$. Replace $\delta,\epsilon,\tau$ with smaller numbers  $\delta',\epsilon',\tau'$ satisfying the same relations as above and $q_1\not \in B_{\tau'}(p_1)$. By repeating the previous procedure we secure another point $q_2'\sim p_2$ that is closer to $p_1$. Continuing in this way, we can build a sequence $(q_n')_{n=1}^\infty$ converging to $p_1$ such that $q_n'\sim p_2$ for all $n$.

We now find our contradiction. Given that both such sequences exist, $f$ cannot be an embedding around $p_1$ nor $p_2$ and has rank $1$ at both points. Choose normal coordinates on $M$ centered at $f(p_1)=f(p_2)$, and a conformal coordinate centered at $p_1$ in which $f$ takes the form $$f^k = p^k + r^k$$ as in the previous subsection. Since $f$ is not regular at $p_1$, there is at least one $k$ such that $\deg p^k = m >1$, $m\neq \infty$. Choosing a conformal coordinate at $p_2$, $f$ takes the form $$f^k = \tilde{p}^k + \tilde{r}^k$$ with $\tilde{p}^k$ a spherical harmonic and $\tilde{r}^k$ decaying faster. The images of $p^k$ and $\tilde{p}^k$ in $\mathbb{R}$ intersect on open sets, so $\tilde{p}^k$ is clearly non-zero. Thus, the set of points near $p_2$ on which $f^k$ is equal to $f^k(p_1)$ is some collection of arcs intersecting at that point. However, since $\deg p^k > 1$, we can find the same contradiction as in Lemma \ref{lonely}.
\end{proof}

 \subsection{The Hausdorff condition}  The main result of this subsection is Proposition \ref{Hausdorff}, which implies the topological quotient of $\Sigma$ by $\sim$ is Hausdorff. We say $p_1\sim' p_2$ if for every pair of neighbourhoods $U_i$ containing $p_i$, there exists $p_i'\in U_i$ with $p_1'\sim p_2'$.
 \begin{prop}\label{Hausdorff}
Suppose $p_1\sim' p_2$. Then $p_1\sim p_2$.
\end{prop}
Proposition \ref{Hausdorff} is our ``holomorphic unique continuation property." Combined with \cite[Theorem 3]{S}, Proposition \ref{Hausdorff} implies the following result of independent interest.
\begin{prop}\label{hUCP}
Let $\mathbb{D}\subset \mathbb{R}^2$ be the unit disk. Suppose $u_1,u_2:\mathbb{D}\to M$ are harmonic maps maps such that, for all open sets $D_1\subset \mathbb{D}$ containing $0$, there is an open subset $D_2\subset\mathbb{D}$ (possibly not containing $0$) such that $u^2(D_2)\subset u^1(D_1)$. Moreover, assume that for any subsets $D_i'\subset D_i$ on which $u_i$ is regular such that $u_2(D_2')\subset u_1(D_1')$, the map $u_2^{-1}|_{u_1(D_1')}\circ u_1|_{D_1'}$ is holomorphic. Then there exists an open subset $D'\subset\mathbb{D}$ containing $0$ such that $u^2(D')\subset u^1(\mathbb{D})$.
\end{prop}
Turning toward the proof of Proposition \ref{Hausdorff}, if $p_1$ and $p_2$ are both not zeros of $\Phi$, then one can follow the argument from the proof of Proposition \ref{equiv}, almost word-for-word, up until the last paragraph. We just need to note that Lemma \ref{nofriends} shows we can choose a $\Phi$-disk surrounding $p_1$ that is small enough that it contains no point equivalent to $p_2$. . 

Going forward, we assume at least one of the two points is a zero of $\Phi$. The main step in the proof is the next lemma.
\begin{lem}\label{step}
There exists $\delta,\tau>0$ such that every $p_1'\in B_\delta(p_1)\backslash\{p_1\}$ is equivalent to a point $p_2'\in B_\tau(p_2)\backslash\{p_2\}$.
\end{lem}
\begin{proof}
Let $\delta,\epsilon>0$ and $\tau = \epsilon+3\delta$. We choose $\delta,\epsilon$ to be small enough such that $B_\delta(p_1)\cap B_\tau(p_2)=\emptyset$ and that in $B_\delta(p_1)\backslash\{p_1\}$ and $B_\tau(p_2)\backslash\{p_2\}$, 
\begin{enumerate}[label=(\roman*)]
    \item we have no points equivalent to the centers, and
    \item there are no zeros of $\Phi$. 
\end{enumerate}
We take open sets $p_1'\in B_\delta(p_1)$, $p_2'\in B_\epsilon(p_2)$ with $p_1'\sim p_2'$, and let $h$ be the associated holomorphic diffeomorphism. Let $q\in B_\delta(p_1)$, $q\neq p_1$, and let $\gamma$ be a path from a point $p_1'$ to $q$. We choose $\gamma$ to be either the straight line from $p_1'$ to $q$, or a slight perturbation of that line to make sure the path does not touch $p_1$. Regardless, we can arrange so the $\Phi$-length is bounded above by $5\delta/2$.

We analytically continue $h$ along $\gamma$ as much as we can. Since the starting point lies in $B_\epsilon(p_2)$, we see the image under $h$ of any segment lies in $B_\tau(p_2)$. If we can continue $h$ along $\gamma$ to the endpoint, and the endpoint of $h(\gamma)$ is not $p_2$, then we have $q=\gamma(1)\sim h(\gamma(1))$. The only way we could not extend is if some segment of $h(\gamma)$ touches $p_2$. Notice that, regardless, we have a point $q\in B_\delta(p_1)$ that satisfies $q\sim' p_2$ (here we are relabelling $q$ to be the endpoint of a bad segment if that happens). We rule this out with the lemma below.
\begin{lem}\label{below}
In the setting above, we can choose our $\Phi$-disks to be small enough so that no point $q\in B_\delta(p_1)\backslash\{p_1\}$ satisfies $q\sim' p_2$.
\end{lem}
\begin{proof}
We first show that given such a point $q$, we have $q\sim' p_1$. Let $U_1,U_2,U_3$ be open sets containing $p_1,p_2,q$ respectively. Let $\delta_1,\delta_2,\delta_3>0$ and find $p_1'\in B_{\delta_1}(p_1)$, $p_2'\in B_{\delta_2}(p_2)$ with $p_1'\sim p_2'$, as well as $p_2''\in B_{\delta_2}(p_2)$, $q'\in B_{\delta_3}(q)$ with $p_2''\sim q'$. We choose the $\delta_j$'s so that $B_{\delta_3+3\delta_2}(q)$ contains no zeros of the Hopf differential, and $B_{\delta_i}(p_i)$ can only have zeros at $p_i$. We also choose the $\delta_i$'s so that all balls mentioned above are contained in $U_1,U_2,U_3$ and disjoint. Let $h$ be the holomorphic map relating $p_2''$ to $q'$. We analytically continue $h$ along a path $\gamma$ from $p_2''$ to $p_2'$ with length at most $5\delta_2/2$ that is chosen to avoid $p_2$. Then the image path lies in $B_{\delta_3+3\delta_2}(q)$ and so we can continue to the endpoint. The endpoint $h(\gamma(1))$ is equivalent to $p_2'$. If the endpoint is not $q$, then $h(\gamma(1))\sim p_2'\sim p_1'$, and this proves the claim. If the endpoint $h(\gamma(1))$ is $q$ itself, then $q\sim p_2'\sim p_1'$, and we can find $q''$ very close to $q$ that is equivalent to a point very close to $p_1'$ (in particular, contained in $B_{\delta_1}(p_1)$). 

Therefore, we see that if the lemma is false, we can construct a sequence $(q_n)_{n=1}^\infty$ converging to $p_1$ such that $q_n\sim' p_1$ for all $n$. Fix a $q_n$, along with a $\delta'>0$ such that $B_{4\delta'}(q_n)$ contains no zeros and no points equivalent to $q_n$ and $B_{\delta'}(p_1)$ has no zeros other than possibly $p_1$. We find $q_n'\in B_{\delta'}(q_n)$ and $p_1'\in B_{\delta'}(p_1)$ such that $q_n'\sim p_1'$. There is another point $q_N\in B_{\delta'}(p_1)$ such that $q_N\sim' p_1$. Connect $p_1'$ to $q_N$ via a path of length at most $5\delta'/2$ that does not touch $p_1$. Analytically continue the associated map $h$ along this path. The image lies in $B_{4\delta'}(q_n)$, so we can always continue. The endpoint $h(\gamma(1))\in B_{4\delta'}(q_n)$ is equivalent to $q_N$. We claim we can choose $q_N$ with the property that $h(\gamma(1))\neq q_n$. To this end, if $h(\gamma(1))=q_N$, we take the straight line path $\sigma$ from $q_N$ to $q_{N+1}$. According to \cite[Theorem 8.1]{St}, if $p_1$ is a zero of $\Phi$ of order $n$, then geodesics in the $\Phi$ metric are either straight lines or the concatenation of two radial lines enclosing an angle of at least $2\pi/(n+2)$. By pigeonholing, we can pass to a subsequence where every $q_n$ lies in a closed sector of angle $\pi/(n+2)$ around the origin. This guarantees that the straight line path from any $q_j$ to $q_k$ is a geodesic in the $\Phi$-metric and has length at most $\delta'$. As $\Phi(q_n)\neq 0$, the image $h(\sigma)$ is then a straight line contained in $B_{4\delta'}(q_n)$ with initial point $q_n$, so it certainly cannot terminate at $q_n$. We prove the claim by replacing $q_N$ with $q_{N+1}$ and taking the concatenation of our original path with the straight line $\sigma$. We now just want to show $q_N\sim q_n$, and we will have a contradiction. Toward this, it is enough to show $q_N\sim' q_n$, since $\Phi$ does not vanish at these points.

This last step is similar to the beginning of our proof, and so we only sketch the argument. Recall that we have $p_1\sim' q_n$ and $p_1\sim' q_N$. Find smalls balls containing $q_n$, $p_1$, and $q_N$. Then within the ball containing $p_1$ we have two points $p_1'$ and $p_1''$, with $p_1'$ equivalent to a point near $q_n$ and $p_1''$ equivalent to a point near $q_N$. Connect $p_1'$ and $p_1''$ via a small arc that does not touch $p_1$. We can arrange for the arc to stay in a ball around $q_n$ in which it can always be continued. We thus get a point near $q_n$ that is equivalent to a point near $q_N$. We may need to wiggle the path so the point is not $q_n$. As discussed above, we are done.
\end{proof}
Returning to the proof of Lemma \ref{step}, we see that we can always extend our chosen segments, and moreoever each $q\in B_\delta(p_1)\backslash\{p_1\}$ has an equivalent point in $B_\tau(p_2)\backslash\{p_2\}$.
\end{proof}
With Lemma \ref{step} in hand, we are now ready to complete the proof of Proposition \ref{Hausdorff}. Let $\Omega_2'$ be the set of points in $B_\tau(p_2)\backslash \{p_2\}$ that have an equivalent point in $B_\delta(p_1)\backslash\{p_1\}$. Let $\Omega_2=\Omega_2'\cup\{p_2\}$. By repeating the previous argument, we can find a very small ball $B_\alpha(p_2)$ such that every point in $B_\alpha(p_2)\backslash\{p_2\}$ is equivalent to a point in $B_\delta(p_1)\backslash\{p_1\}$. This shows that $p_2$ is an interior point of $\Omega_2$. Away from $p_2$, $\Omega_2$ is open by elementary considerations. It is now simple to conclude $p_1\sim p_2$ by using the open sets $B_\delta(p_1)$ and $\Omega_2$.

\subsection{Ramification at branch points} We now investigate the local behaviour of the map $f$ near zeros of the Hopf differential. This leads us to define a notion of ramification for branch points. Our definition is slightly different from the one given in Section 1.
\begin{lem}\label{rot}
Suppose $p$ is a branch point of $f$, and hence a zero of $\Phi$ of some order $n$. Let $h:\Omega_1\to\Omega_2$ be a holomorphic diffeomorphism with $f\circ h =f$, and suppose $\Omega_1,\Omega_2$ are both contained in a ball $B_\epsilon(p)$, where $\epsilon>0$ is chosen so that there are no other zeros and no other point is equivalent to $p$ in $B_{2\epsilon}(p)$. Then, in the natural coordinates for $\Phi$, $h$ is a rational rotation of angle $2\pi j/(n+2)$
\end{lem}
\begin{proof}
Select $p_i\in \Omega_i$ with $h(p_1) = h(p_2)$.  Let $\gamma:[0,1]\to B_\epsilon$ be a straight line path starting at $p_1$ that terminates at the point $p$. We analytically continue $h$ in a simply connected neighbourhood of $\gamma$, as far as we can. Either there is an interior point $q$ in the straight line that is mapped via $h$ to $p$, or we can continue along the whole curve and extend to the boundary point $p$. In the first case, Proposition \ref{Hausdorff} guarantees $q\sim h(q)= p$, which by our choice of $\epsilon$ means $q=p$, contradicting the definition of $q$. In the second case, Proposition \ref{Hausdorff} yields $p\sim h(p)$ and we deduce $h(p)=p$.

We now prove $h$ is a rotation. Work in the interior of the extension of $\Omega_1$ in which $h$ has been continued. If we write the Hopf differential in local coordinates as $\Phi =\phi(z)dz^2$, then $$\phi(z) = \phi(h(z))(h'(z))^2.$$  In the natural coordinate for the Hopf differential this becomes $$z^n = (h(z))^n(h'(z))^2.$$ Since we're in a simply connected region that doesn't touch zero we can choose a branch of the square root. $h$ then satisfies $$z^{n/2}=(h(z))^{n/2}h'(z) = \frac{\partial}{\partial z}\frac{(h(z))^{n/2+1}}{n/2+1}.$$ Integrate to get $$z^{n/2+1} = (h(z))^{n/2+1} + c$$ for some complex constant $c$. Since $h(p)=p$, taking $z\to 0$ along $\gamma$ forces $c=0$. This implies $$z^{n+2}=(h(z))^{n+2}$$ and the result is now clear.
\end{proof}
\begin{defn}
 A non-minimal harmonic map $g$ with Hopf differential $\Phi$ is holomorphically ramified of order $r-1$ if $r$ is the largest integer such that there exists a $\Phi$-disk $\Omega$ centered at $p$ and a holomorphic degree $r$ branched cover $\psi: \Omega\to D$ with one branch point at $p$ onto a disk $D$ with $\psi(p)=0$ and such that $\psi(p_1)=\psi(p_2)$ implies $f(p_1)=f(p_2)$.
\end{defn}
A map is called unramified if $r=1$. Clearly, a map can only ramify non-trivially at a branch point.
\begin{lem}\label{ram}
A non-minimal harmonic map $g$ with Hopf differential $\Phi$ is ramified of order $r>1$ at $p$ if and only if for all $\epsilon>0$, there exists $p_1,p_2\in B_\epsilon(p)\backslash \{p\}$ such that $p_1\sim p_2$ and $p_1\neq p_2$, where $p_1\sim p_2$ in the sense that there is a holomorphic map $h$ taking a neighbourhood of $p_1$ to one of $p_2$ that leaves $g$ invariant.
\end{lem}
\begin{remark}
A similar statement holds for minimal maps. See \cite[Lemma 3.12]{GOR}.
\end{remark}
\begin{proof}
If $g$ is ramified we take a $\Phi$-disk $\Omega$ of $p$ and a map $\psi: \Omega\to D$ as in the definition. Select two points $p_i\neq p$ such that $\psi(p_1)=\psi(p_2)$ as well as neighbourhoods $\Omega_i$ on which $\psi$ is injective and share the same image under $\psi$. Setting $\psi_i=\psi|_{\Omega_i}$, the map $\psi_2^{-1}\circ \psi_1:\Omega_1\to \Omega_2$ is a holomorphic diffeomorphism that leaves $g$ invariant and hence $p_1\sim p_2$. Conversely, pick $\epsilon>0$ such that there are no other zeros of $\Phi$ in $B_{2\epsilon}(p)$ and so we have a coordinate $z$ such that $\Phi = z^ndz^2$. There exists $p_1,p_2\in B_\epsilon(p)$ with $p_1\sim p_2$ but $p_1\neq p_2$. Lemma \ref{rot} shows there are small disks surrounding $p_1,p_2$ that are related by a rotation $h$ of the form $$z\mapsto e^{\frac{2\pi i j}{n+2}}z$$ such that $g=g\circ h$. By the Aronszajn theorem, $g$ is invariant under this rotation in all of $V$. Dividing by the gcd, we see $g$ is invariant under a rotation of the form $$z\mapsto e^{\frac{2\pi i j_1}{r}}z,$$ where $j_1$ and $r$ are coprime. It follows that $g\circ \alpha = g$ in $B_{\epsilon}(p)$, where $\alpha$ is the rotation $z\mapsto e^{2\pi i/r}z$. In these coordinates, we define a holomorphic branched cover $\psi: B_{\epsilon}(p)\to D$ by $\psi(z) = z^{r}$, and note that $\psi(p_1)=\psi(p_2)$ implies $g(p_1)=g(p_2)$.
\end{proof}
\begin{lem}\label{qot}
Let $p$ be a branch point of $f$ of order $m-1$ at which $f$ is ramified of order $r-1$. Then there is a $\Phi$-disk $\Omega$ of $p$ such that $f$ admits a factorization $f|_\Omega=\overline{f}\circ \psi$, where
\begin{enumerate}[label=(\roman*)]
    \item $\psi: \Omega\to D$ is a holomorphic map onto a disk $\{|\zeta|<\delta\}$ such that $\psi|_{\Omega\backslash \{p\}}$ is an $r$-sheeted covering map,
     \item $\overline{f}$ is harmonic with respect to the flat metric on $D$ and the given metric on $M$, and
    \item $\overline{f}:D \to M$ is unramified with a single branch point of order $s-1$ at the origin, where $s=m/r$
\end{enumerate}
\end{lem}
\begin{proof}
Define $\overline{f}$ by $\overline{f}(\psi(z)) =f(z)$. (i) is given and we begin with (ii). Harmonicity is a local matter, and at any point away from zero we can choose a neighbourhood surrounding that point where $\psi^{-1}$ exists and we have the factorization $\overline{f}=f\circ \psi^{-1}$ in that neighbourhood. Since $\psi^{-1}$ is conformal, $\overline{f}$ is harmonic off $0$. Near $0$, we compute $\overline{f}_\zeta$ in coordinates to realise $C^1$ bounds. Via Schauder theory we promote to $C^2$ (or even $C^\infty$) bounds. This implies that the tension field is continuous and therefore vanishes everywhere. As for (iii), we can write each component $f^k$ in certain coordinates as $$f^k = p^k + r^k,$$ where $p^k$ is a spherical harmonic and $r^k$ decays faster than $p^k$. In this form, it is easy to check the branching orders of $f$ and $\overline{f}$.
 
It remains to show that $\overline{f}$ is unramified. Toward this, let $\Theta$ be the Hopf differential of $\overline{f}$ and note the image of a $\Phi$-disk under $\psi$ is a $\Theta$-disk. Indeed, if $\Phi=\phi(z)dz^2$, $\Theta(\zeta) = \theta(\zeta)d\zeta^2$ in local coordinates, then $$\phi(z) = \theta(z^r)\Big ( \frac{\partial z^r}{\partial z}\Big )^2 = \theta(z^r) z^{2r-2}r^{2}.$$ We rearrange to see $$\theta(\zeta) = \theta(z^r) =  z^{n-2r+2}r^{-2}.$$ and the fact that the image is a $\Theta$-disk is derived from direct computation. If $\overline{f}$ is ramified, we can build another holomorphic branched covering map $\psi'$ as in Lemma \ref{ram}. Since both $\psi$ and $\psi'$ have finite fibers, the composition $\psi'\circ \psi$ yields a branched cover of degree greater than $r$, which is impossible. This finishes the proof.
\end{proof}
\begin{remark}
Our computations show that the ramification order is constrained by $r|m$, $r|(n+2)$, and $2r\leq n+2$. The last condition is superfluous, since we always have $2m\leq n+2$. 
\end{remark}
\begin{lem}\label{branch}
For $i=1,2$, let $p_i$ be branch points of $f$ of order $m_i-1$ (we are allowing $m_i=1$), ramified of order $r_i-1$. Then $p_1\sim p_2$ if and only if
\begin{enumerate}[label=(\roman*)]
\item $f(p_1)=f(p_2)$,
    \item $m_1/r_1=m_2/r_2$, and
    \item if $s$ is the common value $m_i/r_i$, there exist maps $\psi_i:U_i\to D$, $\overline{f}_i:D\to M$, $\psi_i(p_i)=0$, such that $\psi_i|_{U_i\backslash\{p\}}$ is an $r_i$-sheeted holomorphic covering map, $f$ factors as $f|_{U_i}=\overline{f}\circ \psi_i$, and $\overline{f}$ is a harmonic map for the flat metric on the disk with a branch point of order $s-1$.
\end{enumerate}
\end{lem}
\begin{proof}
If $m=0$ this is trivial, so we assume $m>0$. Suppose the conditions hold. Given any two open sets $\Omega_i$ containing $p_i$, we can radially shrink our $\Phi$-disks to have $U_i\subset \Omega_i$ (the argument from Lemma \ref{rot} shows any two points with $\psi_i(q_1)=\psi_i(q_2)$ have the same $\Phi$-distance to $p_i$). For $p_1'\in U_1\backslash \{p_1\}$ let $\psi_1'$ be the restriction to a neighbourhood $U_1'$ of $p_1'$ on which $\psi_1$ is injective. Let $\psi_2'$ be the restriction onto some neighbourhood $V_2'$ such that $\psi_2$ maps $U_2'$ injectively onto $\psi_1(U_1')$. Set $p_2'=\psi_2'^{-1}\circ \psi_1'(p_1')$ and $h =\psi_2'^{-1}\psi_1$. $h$ is holomorphic and leaves $f$ invariant. The result follows.

Conversely, assume $p_1\sim p_2$. (i) was already discussed. We first want to show that we can choose $\Phi$-disks $U_i$ that satisfy condition (2) in the definition of $\sim$. We take $\psi_i:U_i\to D_i$ and $f_i: D_i\to M$ as in Lemma \ref{qot}. If $p_1'\in U_1$ is equivalent to $p_2'\in U_2$, then combining our reasoning from Proposition \ref{equiv} with Proposition \ref{Hausdorff} shows $d(p_1,p_1')=d(p_2,p_2')$. We've run this type of argument a few times at this point, but we feel a duty to elaborate. Pick subdisks $U_i'\subset U_i$ that satisfy condition (2) and balls $B_{\delta}(p_1)$, $B_\epsilon(p_2)$ contained in the subdisks, such that in $B_{2\delta}(p_1)$ and $B_{\epsilon+2\delta}(p_2)$ there are no points equivalent to $p_1,p_2$ respectively and no other possible zeros of $\Phi$. Find $p_1'\in B_\delta(p_1)\backslash\{p_1\}$ and $p_2'\in B_\epsilon(p_2)\backslash\{p_2\}$ with $p_1'\sim p_2'$. Take the straight line path $\gamma$ from $p_1'$ to $p_1$ and analytically continue $h$ along $\gamma$ as much as we can. The image of any segment of this path under $h$ is also a straight line contained in $B_{\epsilon+2\delta}(p_2)$. We now have two possibilities:
\begin{enumerate}[label=(\roman*)]
    \item the path $h(\gamma)$ runs into $p_2$ before we have finished extending, or
    \item we can extend $h$ to the boundary point $\gamma(1)=p_1$
\end{enumerate}
In the first scenario, we obtain $d(p_2,p_2')\leq d(p_1,p_1')$. In the latter, Proposition \ref{Hausdorff} ensures $h(\gamma(1))\sim p_1\sim p_2$, so that $h(\gamma(1))=p_2$. Regardless of the situation, we have $$d(p_2',p_2)\leq d(p_1',p_1).$$ To reverse the argument for the other inequality, we go via a straight line from $p_2'$ to $p_2$. For any segment $\gamma'$ along which we can continue, the length of $h(\gamma')$ is now bounded above by $d(p_1',p_1)< \delta$. Thus, we can continue along the whole curve so long as we don't hit $p_1$. In the same way as above we get the opposite inequality. This is the desired result.

Using the definition of $\sim$, we can now assume the $\Phi$-disks $U_i$ are such that $f(U_1)=f(U_2)$ and that for all $p_1'\in V_1$, $p_1' \neq p_1$, there is $p_2'\in V_2$, $p_2'\neq p_2$, such that $p_1'\sim p_2'$, and vice versa. We construct a holomorphic diffeomorphism $G:D_1\to D_2$ such that $$\overline{f}_2\circ G = \overline{f}_1.$$ Let $w_1\in D_1\backslash\{0\}$. We take a small neighbourhood of $w_1$ and a lift to an open set via $\psi_1$ such that the restriction of $\psi_1$ is injective. Let $w_1'$ be the given preimage under $\psi_1$. There is then a point $w_2'\in V_2$ related by a holomorphic map such that $f$ agrees in neighbourhoods surrounding $w_1'$ and $w_2'$. Set $w_2=G(w_1)=\psi_2(w_2')$. We claim there can be no other point with this property. If there was such a $w'$, then we would have $w\sim w'$ with respect to the corresponding equivalence relation for $\overline{f}_2$. However, we know the map $\overline{f}_2$ is unramified, and by Lemma \ref{ram} we can choose our disks small enough that there are no two distinct points in $D_2$ with this property. The association $w_1\mapsto w_1'$ defines our map $G$. If we set $G(0)=0$, then we see $G$ is a diffeomorphism from $D_1\backslash\{0\}\to D_2\backslash\{0\}$, because we can invert the construction. The map $G$ is holomorphic off $\{0\}$. Since it is bounded near $0$, it extends to a holomorphic diffeomorphism on all of $D_1$.

From Lemma \ref{qot} the branching order of $\overline{f}_i$ is $m_i/r_i -1$, and since $G$ is a diffeomorphism, it is clear that these branching orders agree. Defining $D=D_1$ and $\overline{f}$ to be the common map $\overline{f}^2\circ G=\overline{f}_1$, $\psi_1=\overline{\psi}_1$, $\psi_2=G^{-1}\circ \overline{\psi}_2$, (iii) can be verified easily.
\end{proof}

\subsection{Constructing the Riemann surface} Preparations aside, we build the covering space. Our work here is drawn from Propositions 3.19 and 3.24 in \cite{GOR}. Let $\Sigma_0$ denote the space of equivalence classes of $\Sigma$ with respect to $\sim$, equipped with the quotient topology. We denote by $\pi:\Sigma\to \Sigma_0$ the projection map.
\begin{prop}\label{surface}
$\Sigma_0$ is an orientable surface.
\end{prop}

\begin{proof}
For each $p\in M$ let $U$ be a neighbourhood of $p$ with no other point equivalent to $p$ and as in Lemma \ref{qot}, so that we have a map $\psi: U\to D$, a harmonic map $f:D\to M$, and a factorization $f=\overline{f}\circ \psi$. Let $\overline{U}=\{[q]:q\in U\}$. To prove such a set is open, we show any $\pi^{-1}(\overline{U})\subset \Sigma$ is open. If $p_1\in \pi^{-1}(\overline{U})$, then there is $p_2\in V$ such that $p_1\sim p_2$. Then we can find neighbourhoods $\Omega_i$ containing $p_i$ with $\Omega_2\subset U$ and such that for each $p_1'\in\Omega_1$ there exists $p_2'\in\Omega_2$ with $p_1'\sim p_2'$. This implies the $\overline{U}$ define an open cover of $\Sigma_0$.

On each $\overline{U}$ we have a map $\overline{\psi}:\overline{U}\to D$ given by $\overline{\psi}([q])=\psi(q)$. We will see that these maps define charts. If $q_1,q_2\in U$ are such that $q_1\sim q_2$, then $\psi(q_1)$, $\psi(q_2)$ are equivalent with respect to $\overline{f}$ and hence we can choose $U$ so that $\psi(q_1)=\psi(q_2)$, since $\overline{f}$ is unramified. This proves $\overline{\psi}$ is well-defined.

For injectivity, suppose $[p_1],[p_2]\in \overline{U}$ are such that $\overline{\psi}([p_1])=\overline{\psi}([p_2])$. Choosing representatives $p_1,p_2$, either $p_1=p_2=p$ or neither of them is equal to $p$. In the second case, since $\psi$ is a holomorphic covering map on $U\backslash \{p\}$ we can use it to build a holomorphic diffeomorphism from a neighbourhood of $p_1$ to a neighbourhood of $p_2$. Since $f=\overline{f}\circ \psi$ on $U$, this map leaves $f$ invariant.

As for continuity and openness, the argument is the same as the one found in  \cite[page 779]{GOR}. The Hausdorff condition is immediate from Proposition \ref{Hausdorff}. $\Sigma_0$ is orientable because $\pi$ respects the orientation of $\Sigma$.
\end{proof}
There exists a continuous map $f_0:\Sigma_0\to M$ such that $f=f_0\circ \pi$, defined by $f_0([p])=f(p)$.
\begin{prop}\label{hol}
There exists a complex structure on $\Sigma_0$ so that $\pi:\Sigma\to \Sigma_0$ is holomorphic and the map $f_0$ is harmonic with respect to the conformal metric $\mu_0$ obtained via uniformization.
\end{prop}
\begin{proof}
We use the collection of charts specified in Lemma \ref{branch}. Let $(\overline{U}_1,\overline{\psi}_1)$ and $(\overline{U}_2,\overline{\psi}_2)$ be two charts for $\Sigma_0$ arising from open sets $U_1,U_2$ centered at points $p_1,p_2$. We have maps $\psi_i:U_i\to D_i$, $\overline{\psi}_i:\overline{U}_i\to D_i$, $\pi:\Sigma\to \Sigma_0$, and harmonic maps $\overline{f}_i:D_i\to M$ such that $\overline{f}=\overline{f}_i\circ\overline{\psi}_i$, $\psi_i = \overline{\psi}_i\circ \pi$. We show the map $$\overline{\psi}_2\circ \overline{\psi}_1^{-1}:\overline{\psi}_1(\overline{U}_1\cap \overline{U}_2)\subset D_1\to \overline{\psi}_2(\overline{U}_1\cap\overline{U}_2)\subset D_2$$ is holomorphic. 

By the removeable singularities theorem, it suffices to check holomorphy away from the copies of $0$ in $D_i$. Let $[q]\in \overline{U}_1\cap \overline{U}_2$ be so that $\overline{\psi}_i([q])\neq 0$, and choose a neighbourhood $U$ around $[q]$ and $U'\subset \pi^{-1}(U)$ such that
\begin{enumerate}[label=(\roman*)]
\item $0\not \in \overline{\psi}_i(U)$,
    \item the map $\pi|_{U'}:U'\to U$ is injective, and so we can define an inverse $\pi^{-1}:U\to U'$, and
    \item the holomorphic map $\psi_i$ is injective in $U'$, so that we can define a holomorphic inverse $\psi_i^{-1}:\psi_i(U')\to U'$.
\end{enumerate}
Note that $\psi_i(U')=\overline{\psi}_i(U)$. Clearly, the map $\psi_2\circ \psi_1^{-1}$ is holomorphic in $\overline{\psi}_1(U)$. Meanwhile, since we can invert $\pi$, we obtain $$\overline{\psi}_2\circ\overline{\psi}_1^{-1}=(\psi_2\circ \pi_0^{-1})\circ (\psi_1\circ \pi^{-1})^{-1}= \psi_2\circ \psi_1^{-1}.$$ It follows that the map in question is holomorphic near $[q]$, and hence everywhere.

In holomorphic local coordinates, the map $\pi$ is of the form $z\mapsto z$ or $z\mapsto z^n$, so it is surely holomorphic. From conformal invariance of the harmonic map equation, $f_0=\overline{f}_i\circ\overline{\psi}_i$ is harmonic away from images of branch points of $\pi$. The argument of Lemma \ref{qot} shows $f_0$ is globally harmonic.
\end{proof}
This completes the proof of Theorem \ref{main} for holomorphic diffeomorphisms.
\end{section}

\begin{section}{Klein Surfaces} 
We explain the adjustments required to prove Theorem \ref{main} for anti-holomorphic diffeomorphisms $h:\Omega_1\to\Omega_2$.

\subsection{Preparations.} We begin with a review of Klein surfaces. More details on the theory of Klein surfaces can be found in the book \cite{AG}. Set $$\mathbb{C}_+ =\{z\in \mathbb{C}: \textrm{Im}z\geq 0\}$$ to be the closed upper half plane.
\begin{defn}
Let $\Omega\subset \mathbb{C}_+$ be open. A function $f:\Omega\to \mathbb{C}$ is (anti-)holomorphic if there is an open set $U\subset\mathbb{C}$ containing $\Omega$ such that $f$ extends to an (anti-)holomorphic function from $U\to\mathbb{C}$. 
\end{defn}
\begin{defn}
A map between open subsets of $\mathbb{C}$ is dianalytic if its restriction to any component is holomorphic or anti-holomorphic.
\end{defn}
\begin{defn}
 Let $X$ be a topological surface, possibly with boundary. A dianalytic atlas on $X$ is a collection of pairs $\mathcal{U}=\{(U_\alpha,\varphi_\alpha)\}$ where
 \begin{enumerate}[label=(\roman*)]
     \item $U_\alpha$ is an open subset of $X$, $V_\alpha$ is an open subset of $\mathbb{C}_+$, and $\varphi_\alpha : U_\alpha \to V_\alpha$ is a homeomorphism. 
     \item If $U_\alpha\cap U_\beta\neq \emptyset$, the map $$\varphi_\alpha\circ \varphi_\beta^{-1}: \varphi_\beta(U_\alpha \cap U_\beta)\to \varphi_\alpha(U_\alpha\cap U_\beta)$$ is dianalytic.
 \end{enumerate}
A Klein surface is a pair $X=(X,\mathcal{U})$.
\end{defn}
 Closely related is the notion of a Real Riemann surface.
\begin{defn}
A Real Riemann surface is the data $(X,\tau)$ of a Riemann surface $X$ and an anti-holomorphic involution $\tau: X\to X$.
\end{defn}
Given a Real Riemann surface $(X,\tau)$, the quotient $X/\tau$ has the structure of a Klein surface, and as a matter of fact every Klein surface $X$ arises in this fashion (see Chapter 1 in \cite{AG}). The associated Real Riemann surface is called the analytic double, and it is unique up to isomorphism in the category of Real Riemann surfaces. The boundary of the Klein surface corresponds to the fixed-point set of the involution.
\begin{defn}
A harmonic (minimal) map on a Klein surface is a continuous map that lifts to a harmonic (minimal) map on the analytic double with respect to the conformal metric obtained via uniformization.
\end{defn}
To prove Theorem \ref{main} for anti-holomorphic maps, as previously done we define an equivalence relation $\sim$ and build a dianalytic atlas on the topological quotient $\Sigma_0=\Sigma/\sim$. Before we get into details, we make an important reduction:
we apply the holomorphic case of Theorem \ref{main} to $\Sigma$ and acquire a new Riemann surface $\Sigma'$, as well as maps $\pi:\Sigma\to \Sigma'$, $f':\Sigma'\to M$. The key property of the pair $(\Sigma',f')$ is that equivalences classes under Definition \ref{equiv} are singletons.

We define a relation $\sim$ on $\Sigma$ by taking Definition \ref{sim}, but this time insisting the maps involved are merely conformal rather than holomorphic.
\begin{lem}\label{pair}
Given $p\in \Sigma$, there is at most one other point $q\in \Sigma'$ such that $p\sim q$.
\end{lem}
\begin{proof}
Suppose $p,q_1,q_2$ are distinct points and $p\sim q_1$ and $p\sim q_2$. If all points are not in $\Z$, then we have anti-holomorphic maps $h_1,h_2$ relating to $q_1,q_2$ to $p$. The composition $h_2\circ h_1^{-1}$ is then a holomorphic map relating $q_1$ to $q_2$, which means they are equivalent for Definition \ref{equiv}, and this is impossible. If at least one of them is a zero, then we can find disjoint neighbourhoods $\Omega$ containing $p$ and $\Omega_i$ containing $q_i$ such that every point in $\Omega_1\backslash\{q_1\}$ is equivalent to a point in $\Omega\backslash\{p\}$, and every point in $\Omega\backslash\{p\}$ is equivalent to a point in $\Omega_2\backslash\{q_2\}$. This brings us to the non-zero case.
\end{proof}
By the previous lemma, transitivity for $\sim$ holds vacuously. Accordingly, the proof of the lemma below is trivial.
\begin{lem}
$\sim$ is an equivalence relation.
\end{lem}

\subsection{Proof of the main theorem.} Referencing our earlier work, we prove Theorem \ref{main} for anti-holomorphic $h$. Henceforth we abuse notation and set $\Sigma=\Sigma'$, $f=f'$.

The first thing to note is that $h$ is an orientation-reversing isometry for the $\Phi$-metric. Indeed, if $\Phi$ does not vanish on an open subset $U\subset \Omega_1$ and $z$ is a natural coordinate for $\Phi$, then the function $$w=\iota\circ z\circ h^{-1}$$ defines a holomorphic coordinate on $h(U)$, where $\iota$ is the complex conjugation operator on the disk. In this coordinate, $w(h(z))=\overline{z}$, and it can be easily checked that $$df_p\Big ( \frac{\partial}{\partial z}\Big ) = df_{h(p)}\Big ( \frac{\partial}{\partial \overline{w}}\Big )\in T_{f(z)}M\otimes \mathbb{C}.$$ We infer $$\langle f_{\overline{w}}, f_{\overline{w}}\rangle = 1$$ and furthermore $$\langle f_{w}, f_{w}\rangle = \overline{\langle f_{\overline{w}}, f_{\overline{w}}\rangle}= 1.$$  As in Lemma \ref{isometry}, we find that $w$ is a natural coordinate for $\Phi$. The result follows.

Moreover, we can analytically continue $h$ exactly as we did in Proposition \ref{cont}. Moving toward the main proof, we follow the proof of Lemma \ref{step}, word-for-word, and note that Lemma \ref{below} is immediate from Lemma \ref{pair}. The proof of the analogue of Proposition \ref{Hausdorff} follows. As for ramification, we do see new behaviour.

\begin{lem}
Suppose $p$ is a zero of $\Phi$ of order $n\geq 0$. Let $h:\Omega_1\to\Omega_2$ be an anti-holomorphic diffeomorphism with $f\circ h =f$, and so that $\Omega_1,\Omega_2$ are both contained in a ball $B_\epsilon(p)$, where $\epsilon>0$ is chosen so that there are no other zeros and no other point is equivalent to $p$ in $B_{2\epsilon}(p)$. Then, in the natural coordinates for $\Phi$, $$h(z) = e^{\frac{2\pi i j}{n+2}}\overline{z}$$ on its domain.
\end{lem}
\begin{proof}
We follow the proof of Lemma \ref{rot}, except now we have a map $h$ that satisfies $$z^n = (\overline{h}(z))^2 \Big (\frac{\partial \overline{h}}{\partial z}\Big )^2.$$ As in the proof of Lemma \ref{rot}, $h$ is defined in a simply connected open set whose distance to zero can be taken to be arbitrarily small. We observe that $\overline{h}$ is holomorphic, and take a branch of the square root and integrate to derive $$\overline{h}(z) = e^{-\frac{2\pi i j}{n+2}}z$$ for some $j=0,1,\dots, n+1$. We conjugate to finish the proof.
\end{proof}
The lemma implies that in a neighbourhood of a ramification point $p$, $f$ is invariant under the map $$\psi(z) = e^{\frac{2\pi i j}{n+2}} \overline{z}.$$ This is an anti-holomorphic involution that fixes every point on the line $$\mathcal{L}=\{re^{\frac{\pi i j}{n+2}}: -1<r<1\}$$ and acts by reflection across this line on all other points.
\begin{lem}
Let $p$ and $\psi$ be as above. If $\psi(q) =q $, then $q$ has no equivalent points with respect to $\sim$.
\end{lem}
\begin{proof}
$\psi$ is two-to-one in a neighbourhood of $q$. Suppose there exists $q'\in \Sigma$ with $q\sim q'$. Then $q'\not \in B_{\epsilon}(p)$. Using the definition of $\sim$, we can find a small disk $B_{\epsilon'}(q)$ and points $p_1,p_2\in B_{\epsilon'}(q)$ with $p_1\sim p_2$, but we can also find a point $q''$ near $q'$ such that $p_1\sim q''$. This contradicts Lemma \ref{pair}.
\end{proof}
We deduce the following.
\begin{lem}
Every $q\in B_\epsilon(p)\backslash\mathcal{L}$ is equivalent to $\psi(q)$ and only $\psi(q)$.
\end{lem}
We say $f$ anti-holomorphically ramifies near $p$ if $f$ is invariant under an anti-holomorphic involution in a neighbourhood of $p$. In contrast to the holomorphic definition, $f$ can anti-holomorphically ramify near rank $1$ singularities. If $f$ does ramify at $p$, we form the quotient $$\mathcal{K}= B_\epsilon(p)/\psi$$ by identifying points $z$ and $\psi(z)$. This has the structure of a Klein surface with boundary, the boundary being identified with $\mathcal{L}$.

\begin{lem}\label{ramify}
$p\in \Sigma$ satisfies $[p]=\{p\}$ if and only if $f$ ramifies at $p$.
\end{lem}
\begin{proof}
We need only to show that every point at which $f$ is unramified admits an equivalent point. Looking toward a contradiction, suppose there exists $p\in \Sigma$ with $[p]=\{p\}$ and at which $f$ does not ramify and choose $\epsilon>0$ so that no two points are equivalent in $B_\epsilon(p)$ and that there are no zeros of $\Phi$ in $B_{2\epsilon}(p)$.

We claim $[q]=\{q\}$ for every $q\in B_\epsilon(p)$. If not, there is a $q\in B_\epsilon(p)$ that admits an equivalent point $q'\neq q$. Let $h$ be the anti-holomorphic diffeomorphism relating a neighbourhood of $q$ to one of $q'$. In coordinates, analytically continue $h$ along a straight line $\gamma$ from $q$ to $p$. It follows from our assumption $\{p\}=[p]$ that no segment $h(\gamma')$ for $\gamma'\subset \gamma$ can touch $p$, for otherwise we get a point equivalent to $p$. Thus, we can continue to the endpoint, and the endpoint of $h(\gamma)$ is $p$ itself. This implies $$d(p,q)=d(p,q'),$$ which contradicts our choice of $\epsilon>0$, and therefore settles the claim. 

With the claim in hand, we define a map $$\tau: \Sigma \to \Sigma$$ as follows. If $[q]=\{q\}$, set $\tau(q)=q$. If $[q]=\{q,q'\}$, we put $\tau(q)=q'$. If $f$ is unramified at $q$ and $[q]=\{p,q\}$, then $\tau$ is an anti-holomorphic diffeomorphism near p. If $[q]=\{q\}$, then our claim above shows it is the identity map in a neighbourhood of $q$. If $f$ ramifies at $q$, $\tau$ acts like the map $\psi$ considered above. In any event, $\tau$ is real analytic. Since we know the set $\{q: |[q]|=2\}$ is non-empty, $\tau$ is globally anti-holomorphic and moreover cannot fix the point $p$. This gives a contradiction.
\end{proof}

We now come to the main goal. Simply take the anti-holomorphic map $\tau$ defined in the proof above. Checking on a topological base for $\Sigma$, it is clear that $\tau$ is a continuous and open mapping. As $\tau^2=1$, it is an anti-holomorphic diffeomorphism of $\Sigma$. The quotient $$\Sigma_0 =\Sigma/\tau=\Sigma/\sim$$ is the sought Klein surface.
\begin{remark}
We can read off an atlas as follows. If $p$ is not a ramification point, $\sim$ identifies a small neighbourhood of $p$ with no ramification points to some other neighbourhood. The coordinate chart near $p$ then gives the chart on $\Sigma_0$. Transition maps can be holomorphic or anti-holomorphic. If $p$ is a ramification point, the quotient gives us a space $\mathcal{K}$ as above, with two different choices for coordinates: natural coordinates for $\Phi$, or the complex conjugation of those coordinates. Both holomorphic and anti-holomorphic transition maps exist. We omit the technical details.
\end{remark}

With regard to Theorem \ref{main}, we are left to discuss the projection $\pi:\Sigma \to \Sigma_0$ and the harmonic map $f$. The remark gives coordinate expressions for $\pi$ in which we see it is dianalytic. $\Sigma$ is actually the analytic double of $\Sigma_0$, and $f$ clearly descends to a continuous map $f_0$ on $\Sigma_0$ that is harmonic by definition. This finishes the proof of Theorem \ref{main}.

\subsection{Minimal Klein surfaces} For completeness, we extend the work of Gulliver-Osserman-Royden on minimal maps to the anti-holomorphic case. To the author's knowledge, the result of this subsection is new.

We begin with a minimal map $f:(\Sigma,\mu)\to (M,\nu)$ and anti-holomorphic $h:\Omega_1\to \Omega_2$ such that $f\circ h = f$. As in our approach for non-minimal maps, we first apply \cite[Proposition 3.24]{GOR} to assume $\Sigma$ has no points that are holomorphically related. We then define $\sim$ exactly as in subsection 1.1, but allow the diffeomorphisms involved to be conformal. The application of their result assures that Lemma \ref{pair} goes through for $\sim$. The proof of Proposition 3.14 in \cite{GOR} applies to the map $f$, which proves the relation $\sim$ is Hausdorff.

For ramification, the distinction is that $\Phi=0$, so we cannot apply the usual methods. At the same time, all singular points are good branch points. Recall from subsection 1.1 that near a branch point $p$ of order $m$ we can find a neighbourhood of $p$ with a holomorphic coordinate $z$ and coordinates $(x_1,\dots, x_n)$ around $f(p)$ so that $f$ is given by $$x_1=\re z^m \hspace{1mm} , \hspace{1mm} x_2=\im z^m \hspace{1mm}, \hspace{1mm} x_k = \eta_k(z) \hspace{1mm} , \hspace{1mm} k\geq 3,$$ where $\eta_k(z) \in o(|z|^m)$. If we have distinct $p_1,p_2$ in this neighbourhood with $p_1\sim p_2$, then the anti-holomorphic map $h$ that relates the two must satisfy $$(h(z))^m = \overline{z}^m.$$ Consequently, $h$ is of the form $$h(z) = e^{\frac{2\pi i j}{m}}\overline{z}$$ for some $j=0,1,\dots, m-1$. Up until Lemma \ref{ramify}, one can run through subsection 4.2 almost word-for-word. The only difference is that we use coordinate disks rather than natural coordinates for a holomorphic differential. The analogue of Lemma \ref{ramify} can be worked out without difficulty.
\begin{lem}
In this setting, $p\in \Sigma$ satisfies $[p]=\{p\}$ if and only if $f$ ramifies at $p$.
\end{lem}
\begin{proof}
Even if $f$ is minimal, analytic continuation is possible. Given a curve $\gamma$ starting in $\Omega_1$, we can analytically continue $h$ along $\gamma$ as long as $\gamma$ and $h(\gamma)$ stay sufficiently far away from the set $$\{p\in \Sigma: [p]\textrm{ intersects the branch set of }f\}.$$ To do so, we first can assume $f$ is a diffeomorphism on $\Omega_i$ and injective on $\overline{\Omega_i}$. If $q$ is the first point at which $\gamma$ strikes $\partial\Omega_1$, then $h(q)$ is well-defined. We choose disks $U_1$ and $U_2$ around $q$ and $h(q)$ respectively such that $f|_{\overline{U_i}}$ is a diffeomorphism. We then invoke the unique continuation property of Gulliver-Osserman-Royden to find a smaller disk $U_1'\subset U_1$ such that $f(U_1')\subset U_2$. Setting $U_2'=f|_{U_2}^{-1}(f(U_1'))$, the map $$f|_{U_2'}^{-1}\circ f|_{U_1'}: U_1'\to U_2'$$ is a conformal diffeomorphism that continues $h$, and is therefore anti-holomorphic. This establishes the continuation result. We also note that \cite[Proposition 3.14]{GOR} implies that if $\gamma$ is a curve along which we have continued $h$, then $p\sim h(p)$ for all $p$ in the image of $\gamma$. 

 We suppose there is a point $p$ at which $f$ is unramified and such that $[p]=\{p\}$. Choose a coordinate disk $\Omega$ around $p$ in which no two points are equivalent. We show that under this assumption we must have $[q]=\{q\}$ for all $q\in \Omega$. If not, then there is a $q\in \Omega$ and a $q'\not \in \Omega$ such that $q\sim q'$, and an anti-holomorphic diffeomorphism $h$ relating a neighbourhood of $q$ to one of $q'$. We analytically continue $h$ along a simple curve from $q$ to $p$ that does not touch any point that is equivalent to a branch point of $f$. It is easy to build such a curve, since the branch set is discrete, and equivalence classes can have only two points. Using the reasoning from Lemma \ref{ramify}, we can continue along all of $\gamma$ and $h(\gamma(1))=p$. Now, note that by assumption there is no pair $p_1,p_2\in B_\epsilon(p)$ with $p_1\in \gamma([0,1])$ and $p_1\sim p_2$. Taking $\gamma$ to the endpoint gives that $h(\gamma(t))$ lies outside $B_\epsilon(p)$ for $t\in [0,1]$ sufficiently close to $1$. This contradicts $h(p)=p$, and hence yields $[q]=\{q\}$ for all $q\in \Omega$. We can now conclude the proof exactly as we did in Lemma \ref{ramify}.
\end{proof}
The remainder of the content in subsection 4.2 goes through verbatim. The resulting map from the Klein surface to $M$ is minimal.
\end{section}

\bibliographystyle{alpha}
\bibliography{bibliography}

\begin{thebibliography}{GOR73}

\bibitem[AG71]{AG}
Norman~L. Alling and Newcomb Greenleaf.
\newblock {\em Foundations of the theory of {K}lein surfaces}.
\newblock Lecture Notes in Mathematics, Vol. 219. Springer-Verlag, Berlin-New
  York, 1971.

\bibitem[Aro57]{Ar}
N.~Aronszajn.
\newblock A unique continuation theorem for solutions of elliptic partial
  differential equations or inequalities of second order.
\newblock {\em J. Math. Pures Appl. (9)}, 36:235--249, 1957.

\bibitem[Che76]{C}
Shiu~Yuen Cheng.
\newblock Eigenfunctions and nodal sets.
\newblock {\em Comment. Math. Helv.}, 51(1):43--55, 1976.

\bibitem[CM11]{CM}
Tobias~Holck Colding and William~P. Minicozzi, II.
\newblock {\em A course in minimal surfaces}, volume 121 of {\em Graduate
  Studies in Mathematics}.
\newblock American Mathematical Society, Providence, RI, 2011.

\bibitem[GOR73]{GOR}
R.~D. Gulliver, II, R.~Osserman, and H.~L. Royden.
\newblock A theory of branched immersions of surfaces.
\newblock {\em Amer. J. Math.}, 95:750--812, 1973.

\bibitem[Gul73]{G}
Robert~D. Gulliver, II.
\newblock Regularity of minimizing surfaces of prescribed mean curvature.
\newblock {\em Ann. of Math. (2)}, 97:275--305, 1973.

\bibitem[HW53]{HW}
Philip Hartman and Aurel Wintner.
\newblock On the local behavior of solutions of non-parabolic partial
  differential equations.
\newblock {\em Amer. J. Math.}, 75:449--476, 1953.

\bibitem[Jos08]{Jo2}
J\"urgen Jost.
\newblock {Harmonic mappings}.
\newblock In Lizhen Ji, Peter Li, Richard Schoen, and Leon Simon, editors, {\em
  Handbook of geometric analysis. Vol. 1}, volume~7 of {\em Advanced lectures
  in mathematics (International Press)}, pages 147--194. International Press of
  Boston, Boston, 2008.

\bibitem[JY83]{Jo}
J\"{u}rgen Jost and Shing~Tung Yau.
\newblock Harmonic mappings and {K}\"{a}hler manifolds.
\newblock {\em Math. Ann.}, 262(2):145--166, 1983.

\bibitem[Moo06]{M1}
John~Douglas Moore.
\newblock Bumpy metrics and closed parametrized minimal surfaces in
  {R}iemannian manifolds.
\newblock {\em Trans. Amer. Math. Soc.}, 358(12):5193--5256, 2006.

\bibitem[Moo17]{M2}
John~Douglas Moore.
\newblock {\em Introduction to global analysis}, volume 187 of {\em Graduate
  Studies in Mathematics}.
\newblock American Mathematical Society, Providence, RI, 2017.
\newblock Minimal surfaces in Riemannian manifolds.

\bibitem[MS12]{MS}
Dusa McDuff and Dietmar Salamon.
\newblock {\em {$J$}-holomorphic curves and symplectic topology}, volume~52 of
  {\em American Mathematical Society Colloquium Publications}.
\newblock American Mathematical Society, Providence, RI, second edition, 2012.

\bibitem[MW95]{MW}
Mario~J. Micallef and Brian White.
\newblock The structure of branch points in minimal surfaces and in
  pseudoholomorphic curves.
\newblock {\em Ann. of Math. (2)}, 141(1):35--85, 1995.

\bibitem[Oss70]{O}
Robert Osserman.
\newblock A proof of the regularity everywhere of the classical solution to
  {P}lateau's problem.
\newblock {\em Ann. of Math. (2)}, 91:550--569, 1970.

\bibitem[Sam78]{S}
J.~H. Sampson.
\newblock Some properties and applications of harmonic mappings.
\newblock {\em Ann. Sci. \'{E}cole Norm. Sup. (4)}, 11(2):211--228, 1978.

\bibitem[Str84]{St}
Kurt Strebel.
\newblock {\em Quadratic differentials}, volume~5 of {\em Ergebnisse der
  Mathematik und ihrer Grenzgebiete (3) [Results in Mathematics and Related
  Areas (3)]}.
\newblock Springer-Verlag, Berlin, 1984.

\bibitem[Woo74]{Wood}
John~C. Wood.
\newblock {\em Harmonic maps between surfaces}.
\newblock Ph.D. Thesis. University of Warwick, 1974.

\bibitem[Woo77]{wood2}
John~C. Wood.
\newblock Singularities of harmonic maps and applications of the
  {G}auss-{B}onnet formula.
\newblock {\em Amer. J. Math.}, 99(6):1329--1344, 1977.

\end{thebibliography}
\end{document}